\documentclass[11pt, notitlepage]{article}   

\usepackage{amsmath,amsthm,amsfonts}   

\usepackage{amscd}
\usepackage{amsfonts}
\usepackage{amssymb}
\usepackage{color}      
\usepackage{epsfig}
\usepackage{graphicx}           

\usepackage{epsfig}
\include{psfig}

\theoremstyle{plain}
\newtheorem{theorem}{Theorem}[section]
\newtheorem{lemma}[theorem]{Lemma}
\newtheorem{cor}[theorem]{Corollary}
\newtheorem{prop}[theorem]{Proposition}
\newtheorem{obs}[theorem]{Observation}
\newtheorem{rem}[theorem]{Remark}

\theoremstyle{definition}
\newtheorem{dfn}[theorem]{Definition}

\def\finf{\mathop{{\rm I}\kern -.27 em {\rm F}}\nolimits}

\begin{document}

\title{Metric Dimension and Zero Forcing Number\\ of Two Families of Line Graphs}

\author{{\bf Linda Eroh$^1$}, {\bf Cong X. Kang$^2$}, and {\bf Eunjeong Yi$^3$}\\
\small$^1$ University of Wisconsin Oshkosh, Oshkosh, WI 54901, USA\\
\small $^{2,3}$Texas A\&M University at Galveston, Galveston, TX 77553, USA\\
$^1${\small\em eroh@uwosh.edu}; $^2${\small\em kangc@tamug.edu}; $^3${\small\em yie@tamug.edu}}

\maketitle

\date{}

\begin{abstract}
\noindent Zero forcing number has recently become an interesting graph parameter studied in its own right since its introduction by the ``AIM Minimum Rank -- Special Graphs Work Group", whereas metric dimension is a well-known graph parameter. We investigate the metric dimension and the zero forcing number of some line graphs by first determining the metric dimension and the zero forcing number of the line graphs of wheel graphs and the bouquet of circles. We prove that $Z(G) \le 2Z(L(G))$ for a simple and connected graph $G$. Further, we show that $Z(G) \le Z(L(G))$ when $G$ is a tree or when $G$ contains a Hamiltonian path and has a certain number of edges. We compare the metric dimension with the zero forcing number of a line graph by demonstrating a couple of inequalities between the two parameters. We end by stating some open problems.
\end{abstract}


\section{Introduction}

Let $G = (V(G),E(G))$ be a finite, simple, undirected, and connected graph of order $|V(G)| \ge 2$ and size $|E(G)|$. For a given graph $G$ and $S \subseteq V(G)$, we denote by $\langle S \rangle$ the subgraph induced by $S$. For a vertex $v \in V(G)$, the \emph{open neighborhood of $v$} is the set $N(v)=\{u \mid uv \in E(G)\}$, and the \emph{degree} of a vertex $v \in V(G)$ is $\deg_G(v)=|N(v)|$; an \emph{end-vertex} (also called \emph{pendant}) is a vertex of degree one. We denote by $\Delta(G)$ the maximum degree, and by $\delta(G)$ the minimum degree of a graph $G$. The \emph{distance} between two vertices $u, v \in V(G)$, denoted by $d_G(u, v)$, is the length of the shortest path in $G$ between $u$ and $v$; we omit $G$ when ambiguity is not a concern. The diameter, $diam(G)$, of a graph G is given by $\max \{d(u, v) \mid  u, v \in V (G) \}$. The \emph{line graph} $L(G)$ of a simple graph $G$ is the graph whose vertices are in one-to-one correspondence with the edges of $G$; two vertices of $L(G)$ are adjacent if and only if the corresponding edges of $G$ are adjacent. Whitney \cite{Whitney} showed that $K_3$ and $K_{1,3}$ are the only two connected non-isomorphic graphs having the same line graph. \\

A vertex $x \in V(G)$ \emph{resolves} a pair of vertices $u,v \in V(G)$ if $d(u,x) \neq d(v,x)$. A set of vertices $W \subseteq V(G)$ \emph{resolves} $G$ if every pair of distinct vertices of $G$ is resolved by some vertex in $W$; then $W$ is called a \emph{resolving set} of $G$. For an ordered set $W=\{w_1, w_2, \ldots, w_k\} \subseteq V(G)$ of distinct vertices, the \emph{metric code} (or \emph{code}, for short) of $v \in V(G)$ with respect to $W$, denoted by $code_W(v)$, is the $k$-vector $(d(v, w_1), d(v, w_2), \ldots, d(v, w_k))$. The \emph{metric dimension} of $G$, denoted by $\dim(G)$, is the minimum cardinality over all resolving sets of $G$. Slater \cite{Slater2, Slater} introduced the concept of a resolving set for a connected graph under the term \emph{locating set}. He referred to a minimum resolving set as a \emph{reference set}, and the cardinality of a minimum resolving set as the \emph{location number} of a graph. Independently, Harary and Melter in~\cite{HM} studied these concepts under the term \emph{metric dimension}. Since metric dimension is suggestive of the dimension of a vector space in linear algebra, sometimes a minimum resolving set of $G$ is called a \emph{basis} of $G$. Metric dimension as a  graph parameter has numerous applications; among them are robot navigation \cite{landmarks}, sonar \cite{Slater}, combinatorial optimization \cite{MathZ}, and pharmaceutical chemistry \cite{CEJO}. It is noted in \cite{NPcompleteness} that determining the metric dimension of a graph is an NP-hard problem. Metric dimension has been heavily studied; for surveys, see \cite{base} and \cite{MDsurvey}.\\

The notion of a zero forcing set, as well  as the associated zero forcing number, of a simple graph was introduced in~\cite{AIM} to bound the minimum rank of associated matrices for numerous families of graphs. Let $mr(G)$ be the minimum rank and let $M(G)$ be the maximum nullity of the associated matrices of a graph $G$; then $mr(G)+M(G)=|V(G)|$. Let each vertex of a graph $G$ be given one of two colors, ``black" and ``white" by convention. Let $S$ denote the (initial) set of black vertices of $G$. The \emph{color-change rule} converts the color of a vertex $u_2$ from white to black if the white vertex $u_2$ is the only white neighbor of a black vertex $u_1$; we say that $u_1$ forces $u_2$, which we denote by $u_1 \rightarrow u_2$. And a sequence, $u_1 \rightarrow u_2 \rightarrow \cdots \rightarrow u_{i} \rightarrow u_{i+1} \rightarrow \cdots \rightarrow u_t$, obtained through iterative applications of the color-change rule is called a \emph{forcing chain}. Note that, at each step of the color change, there may be two or more vertices capable of forcing the same vertex. The set $S$ is said to be \emph{a zero forcing set} of $G$ if all vertices of $G$ will be turned black after finitely many applications of the color-change rule. The \emph{zero forcing number} of $G$, denoted by $Z(G)$, is the minimum of $|S|$ over all zero forcing sets $S \subseteq V(G)$. It is shown in \cite{AIM} that $M(G) \le Z(G)$. The zero forcing parameter has been heavily studied; see \cite{ZFsurvey, ZFsurvey2} for surveys. More recently, the zero forcing parameter has become a graph parameter of interest studied in its own right~\cite{iteration, proptime}.\\

Bailey and Cameron initiated a comparative study of metric dimension and base size (along with other invariants) of a graph in~\cite{base}. In~\cite{dimZ}, we initiated  a comparative study between metric dimension and zero forcing number of graphs. The metric dimension and the zero forcing number coincide for paths $P_n$, cycles $C_n$, complete graphs $K_n$, complete bi-partite graphs $K_{s,t}$ ($s+t \ge 3$), for examples; they are $1$, $2$, $n-1$, and $s+t-2$, respectively. For the Cartesian product of two paths, zero forcing number can be seen to be arbitrarily larger than the metric dimension. On the other hand, the bouquet (or amalgamation) of circles shows that the metric dimension may be arbitrarily larger than the zero forcing number (see \cite{iteration} and \cite{bouquetD}). Recently, Feng, Xu and Wang \cite{line} obtained the bounds of the metric dimension of the line graph $L(G)$ of a connected graph $G$ of order at least five, and they proved that $\dim(L(T))=\dim(T)$ for a tree $T$. In this paper, we determine the metric dimension and the zero forcing number of some line graphs. We show that $\dim(L(W_{1,n}))=n-\lceil \frac{n}{3}\rceil$ for a wheel graph $W_{1,n}=C_n+K_1$, where $n \ge 6$, and $\dim(L(B_n))=2n-1$ for a bouquet $B_n$ of $n \ge 2$ circles. We prove that $Z(G) \le 2Z(L(G))$ for a simple and connected graph $G$. Also, we prove that $Z(L(W_{1,n}))=n+1$ for $n \ge 3$, and $Z(L(B_n))=2n-1$ for $n \ge 2$. Further, we show that $Z(G) \le Z(L(G))$ when $G$ is a tree or when $G$ contains a Hamiltonian path and has a certain number of edges. Finally, we compare the metric dimension with the zero forcing number of a line graph by demonstrating a couple of inequalities between the two parameters. We conclude this paper with some open problems.


\section{Metric Dimension of Some Line Graphs}

To put things in perspective, before proceeding onto results specific to our paper, we recall some basic facts on the metric dimension of graphs.\\

\begin{theorem} \cite{CEJO} \label{dimbounds}
For a connected graph $G$ of order $n \ge  2$ and diameter $d$, $$f (n, d) \le \dim(G) \le n-d,$$
where $f (n, d)$ is the least positive integer $k$ for which $k + d^k \ge n$.
\end{theorem}

\begin{theorem} \cite{line} \label{linebounds}
For a connected graph $G$ of order $n \ge 5$,
$$\lceil \log_2 \Delta (G) \rceil \le \dim(L(G)) \le n-2.$$
\end{theorem}

A generalization of Theorem \ref{dimbounds} has been given in \cite{new} by Hernando et al.

\begin{theorem} \cite{new}
Let $G$ be a graph of order $n$, diameter $d \ge 2$, and metric dimension $k$. Then
$$n \le \left(\left\lfloor \frac{2d}{3}\right\rfloor+1\right)^k+k\sum_{i=1}^{\lceil \frac{d}{3} \rceil} (2i-1)^{k-1}.$$
\end{theorem}

The following definitions are stated in \cite{CEJO}. Fix a graph $G$. A vertex of degree at least three is called a \emph{major vertex}. An end-vertex $u$ is called \emph{a terminal vertex of a major vertex} $v$ if $d(u, v)<d(u, w)$ for every other major vertex $w$. The \emph{terminal degree} of a major vertex $v$ is the number of terminal vertices of $v$. A major vertex $v$ is an \emph{exterior major vertex} if it has positive terminal degree. Let $\sigma(G)$ denote the number of end-vertices of $G$, and let $ex(G)$ denote the number of exterior major vertices of $G$.

\begin{theorem}\cite{CEJO, landmarks, PoZh} \label{tree}
If $T$ is a tree that is not a path, then $\dim(T)=\sigma(T)-ex(T)$.
\end{theorem}

\begin{theorem} \cite{line} \label{lineT}
For a tree $T$ that is not a path, $\dim(L(T))=\sigma(T)-ex(T)$.
\end{theorem}

\begin{theorem} \cite{CEJO} \label{dimthm}
Let $G$ be a connected graph of order $n \ge 2$. Then
\begin{itemize}
\item[(a)] $\dim(G)=1$ if and only if $G=P_n$,
\item[(b)] $\dim(G)=n-1$ if and only if $G=K_n$,
\item[(c)] for $n \ge 4$, $\dim(G)=n-2$ if and only if $G=K_{s,t}$ ($s,t \ge 1$), $G=K_s + \overline{K}_t$ ($s \ge 1, t \ge 2$), or $G=K_s + (K_1 \cup K_t)$ ($s, t \ge 1$); here, $A+B$ denotes the join of two graphs $A$ and $B$,
and $\overline{C}$ denotes the complement of a graph $C$.
\end{itemize}
\end{theorem}

\begin{theorem} \cite{base}
For the complete graph $K_n$ for $n \ge 6$, $\dim(L(K_n))=\lceil \frac{2n}{3}\rceil$.
\end{theorem}


\subsection{Wheel Graphs}

Let $W_{1,n}=C_n + K_1$ be the wheel graph on $n+1$ vertices (see (A) of Figure \ref{lemmawheel}). 

\begin{theorem} \cite{wheel2,wheels} \label{wheel}
For $n \ge 3$, let $W_{1,n}=C_n+K_1$ be the wheel graph on $n+1$ vertices. Then
\begin{equation} \nonumber
\dim(W_{1,n}) = \left\{
\begin{array}{cl}
3 & \mbox{if }n=3 \mbox{ or }n=6, \\
\lfloor \frac{2n+2}{5}\rfloor & \mbox{otherwise} .
\end{array} \right.
\end{equation}
\end{theorem}

\begin{figure}[htbp]
\begin{center}
\scalebox{0.37}{\input{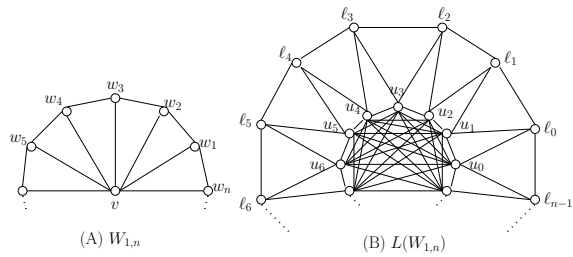}} \caption{\small{The wheel graph and its line graph}}\label{lemmawheel}
\end{center}
\end{figure}

\begin{theorem}\label{propdimlinewheel}
For $n \ge 6$, $\dim(L(W_{1,n})) = n- \lceil \frac{n}{3} \rceil$.
\end{theorem}

\begin{proof}
Let $S$ be a resolving set for $L(W_{1,n})$, where $n \ge 6$.\\

First, we show that $\dim(L(W_{1,n})) \le n- \lceil \frac{n}{3} \rceil$ by constructing a resolving set for $L(W_{1, n})$ of cardinality $n- \lceil \frac{n}{3} \rceil$. See (B) of Figure \ref{lemmawheel} for the labeling of $L(W_{1,n})$. We consider three cases.

\vspace{.1in}

\emph{Case 1: $n=3k$, where $k \ge 2$.} One can easily check that $S=\{\ell_i \mid i \equiv 1 \mbox{ or } 2 \pmod 3\}$ forms a resolving set for $L(W_{1, 3k})$ with $|S|=2k$.

\vspace{.1in}

\emph{Case 2: $n=3k+1$, where $k \ge 2$.} One can easily check that $S=\{\ell_i \mid i \equiv 1 \mbox{ or } 2 \pmod 3\}$ forms a resolving set for $L(W_{1, 3k+1})$ with $|S|=2k$.

\vspace{.1in}

\emph{Case 3: $n=3k+2$, where $k \ge 2$.} One can easily check that $S=\{\ell_i \mid i \equiv 1 \mbox{ or } 2 \pmod 3 \mbox{ and } 0 \le i \le 3k-1\} \cup \{\ell_{3k}\}$ forms a resolving set for $L(W_{1, 3k+2})$ with $|S|=2k+1$.

\vspace{.1in}

Thus, $\dim(L(W_{1,n})) \le n- \lceil \frac{n}{3} \rceil$ for $n \ge 6$.\\

Next, we will show that $\dim(L(W_{1,n})) \ge n- \lceil \frac{n}{3} \rceil$. For $0 \leq i \leq n-1$, define $U_i = \{\ell_{i-1},u_i,\ell_{i}\}$, where the subscript is taken modulo $n$. For each $i$, define $c_i$ as follows. We will use $c_i$ to count the vertices in $S \cap U_i$, where $\ell_i$ and $\ell_{i-1}$ are each counted as $\frac{1}{2}$, because each of these vertices might possibly appear in two different sets, and $u_i$ is $1$. Thus, if $S \cap U_i = \emptyset$, then $c_i=0$. If $S \cap U_i = \{\ell_{i-1}\}$ or $\{\ell_{i}\}$, then $c_i = \frac{1}{2}$. If $S \cap U_i = \{u_i\}$ or $\{\ell_{i-1},\ell_{i}\}$, then $c_i = 1$. If $S \cap U_i =\{u_i, \ell_i\}$ or $\{u_i,\ell_{i-1}\}$, then $c_i = 1.5$. When $U_i \cap S = U_i$, we have $c_i = 2$.  Notice that $|S| \geq \sum_{i=0}^{n-1} c_i$.

\vspace{.1in}

First, we claim that $c_i=0$ for at most one value of $i$. Suppose, to the contrary, that $c_i=0$ and $c_j=0$ where $i \neq j$. Then $u_i$ and $u_j$ are not resolved by $S$, since $d(u_i,\ell_t)=d(u_j,\ell_t)=2$ for $t \neq i-1$, $i$, $j-1$ or $j$, and $d(u_i,u_t)=d(u_j,u_t)=1$ for $t \neq i$ or $j$. Thus, $c_i =0$ for at most one $i$.

\vspace{.1in}

Next, suppose $c_i = \frac{1}{2}$ for some $i$. Suppose $U_i \cap S = \{\ell_i\}$.  If $U_{i+1} \cap S = \{\ell_i\}$, then $u_i$ and $u_{i+1}$ are not resolved, since $d(u_i,\ell_i)=d(u_{i+1},\ell_i)=1$, $d(u_i,\ell_t)=d(u_{i+1},\ell_t)=2$ for $t \neq i-1$, $i$, or $i+1$, and $d(u_i,u_t)=d(u_{i+1},u_t)=1$ for $t \neq i$ or $i+1$. Thus, $c_{i+1} \geq 1$. Similarly, if $U_i \cap S = \{\ell_{i-1}\}$, then $c_{i-1} \geq 1$. Thus, if $c_i = \frac{1}{2}$ for some $i$, then either $c_{i-1}$ or $c_{i+1}$ is at least $1$.

\vspace{.1in}

If $n = 3k+1$, then it follows from the above observations that $|S| \geq \sum_{i=0}^{n-1} c_i \geq k(\frac{1}{2} + 1 + \frac{1}{2}) + 0 = 2k = n - \left\lceil\frac{n}{3}\right\rceil$. If $n = 3k+2$, then $|S| \geq \sum_{i=0}^{n-1} c_i \geq k(\frac{1}{2} + 1 + \frac{1}{2}) + 1 + 0 = 2k + 1 = n - \left\lceil\frac{n}{3}\right\rceil$. If $n = 3k$, then $|S| \geq \sum_{i=0}^{n-1} c_i \geq (k-1)(\frac{1}{2} + 1 + \frac{1}{2}) + 1 + \frac{1}{2} + 0 = 2k - \frac{1}{2}$. Since $|S|$ is an integer, we have $|S| \geq 2k = n - \left\lceil\frac{n}{3}\right\rceil$.

\vspace{.1in}

Thus, $\dim(L(W_{1,n})) \ge n- \lceil \frac{n}{3} \rceil$ for $n \ge 6$. Therefore, $\dim(L(W_{1,n})) = n- \lceil \frac{n}{3} \rceil$ for $n \ge 6$. \hfill
\end{proof}

\begin{figure}[htbp]
\begin{center}
\scalebox{0.36}{\input{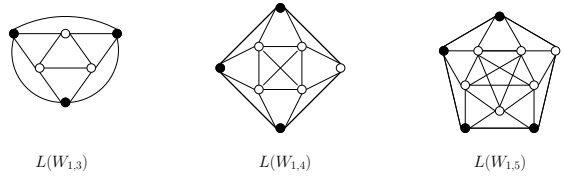}} \caption{\small{Shown in solid vertices form a minimum resolving set for $L(W_{1,n})$, for each $n \in \{3,4,5\}$}}\label{LW3}
\end{center}
\end{figure}

\begin{prop}\label{dimwheel5}
$\dim(L(W_{1,3}))=\dim(L(W_{1,4}))=3$ and $\dim(L(W_{1,5}))=4$.
\end{prop}

\begin{proof}
We first consider $n=3,4$ (see Figure \ref{LW3}). If $n=3$, then $\dim(L(W_{1,3}))>2$ since any two vertices $u, v$ in $L(W_{1,3})$ satisfies $|N(v) \cap N(w)| \ge 2$, and $\dim(L(W_{1,3})) \le 3$ since $\{\ell_0, \ell_1, \ell_2\}$ forms a resolving set for $L(W_{1,3})$; thus, $\dim(L(W_{1,3}))=3$. If $n=4$, then $diam(L(W_{1,4}))=2$ and $\dim(L(W_{1,4}))=3$: $\dim(L(W_{1,4})) \ge 3$ by Theorem \ref{dimbounds}, and $\dim(L(W_{1,4})) \le 3$ since $\{\ell_0, \ell_1, \ell_2\}$ forms a resolving set for $L(W_{1,4})$. Next, we consider $n=5$; notice that $\dim(L(W_{1,5})) \ge 3$ by Theorem~\ref{dimbounds} and the fact that $diam(L(W_{1,5}))=2$. One can easily check that $\{\ell_0, \ell_1, \ell_2, \ell_3\}$ forms a resolving set for $L(W_{1,5})$, and hence $\dim(L(W_{1,5})) \le 4$. Further, a case by case analysis, based on the sizes of $|S\cap \{\ell_i \mid 0 \le i \le 4\}|$ and $|S\cap \{u_i \mid 0 \le i \le 4\}|$, shows that $\dim(L(W_{1,5}))>3$. Therefore, $\dim(L(W_{1,5}))=4$.\hfill
\end{proof}

By Theorem \ref{propdimlinewheel} and Proposition \ref{dimwheel5}, we have the following

\begin{cor} \label{theoremdimwheel}
For $n \ge 3$,
\begin{equation}\nonumber
\dim(L(W_{1,n}))= \left\{
\begin{array}{ll}
3 & \mbox{ if } n=3,4 \\
4 & \mbox{ if }n=5\\
n-\lceil \frac{n}{3}\rceil & \mbox{ if } n \ge 6 .
\end{array} \right.
\end{equation}
\end{cor}


\subsection{Bouquet of Circles}

The bouquet of circles has been studied as a motivating example to introduce the fundamental group on a graph (see
p.189, \cite{massey}). More recently, Llibre and Todd \cite{bouquet}, for instance, studied a class of maps on a bouquet of circles from a dynamical
system perspective.\\

It is well known that $\dim(C_n)=2$ for $n \ge 3$. For $k_n \ge k_{n-1} \ge \ldots \ge k_2 \ge k_1 \ge 2$, let $B_n=(k_1+1, k_2+1, \ldots, k_n+1)$ be a bouquet of $n \ge 2$ circles $C^1$, $C^2$, $\ldots$, $C^n$, with the cut-vertex $v$, where $k_i+1$ is the number of vertices of $C^i$ ($1 \le i \le n$). Let $V(C^i)=\{v, w_{i,1}, w_{i,2}, \ldots, w_{i, k_i}\}$ such that $vw_{i,1} \in E(B_n)$ and $vw_{i, k_i} \in E(B_n)$, and let the vertices in $C^i$ be cyclically labeled, where $1 \le i \le n$. See Figure \ref{bouquetlabeling} for $B_4=(3,4,5,6)$ and its line graph.\\

\begin{figure}[htbp]
\begin{center}
\scalebox{0.37}{\input{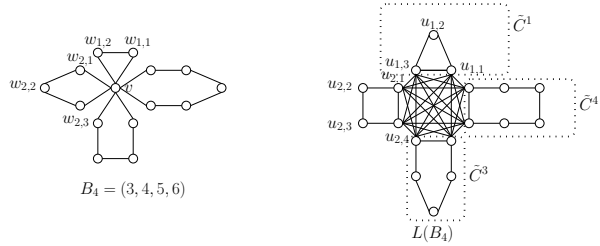}} \caption{\small{Labelings of a bouquet of four circles $B_4=(3,4,5,6)$ and its line graph}}\label{bouquetlabeling}
\end{center}
\end{figure}

\begin{theorem} \cite{bouquetD}
Let $B_n=(k_1+1, k_2+1, \ldots, k_n+1)$ be a bouquet of $n \ge 2$ circles with
a cut-vertex. If $x$ is the number of even cycles of $B_n$, then
\begin{equation*}\label{bouquetdim}
\dim(B_n)= \left\{
\begin{array}{ll}
n & \mbox{ if } x=0 \\
n+x-1 & \mbox{ if } x \ge 1 .
\end{array} \right.
\end{equation*}
\end{theorem}

Referring to Figure \ref{bouquetlabeling}, let $\tilde{C}^1=\{u_{1, j} \mid 1 \le j \le k_1\}, \tilde{C}^2=\{u_{2, j} \mid 1 \le j \le k_2\}, \ldots, \tilde{C}^n=\{u_{n, j} \mid 1 \le j \le k_n\}$ such that $\langle \tilde{C}^i\rangle=L(\langle V(C^i) \rangle)$ for $C^i \subseteq B_n$, where $1 \le i \le n$ and $n \ge 2$.

\begin{theorem} \label{DimLBouquet}
Let $B_n=(k_1+1, k_2+1, \ldots, k_n+1)$ be a bouquet of $n \ge 2$ circles with a cut-vertex. Then $\dim(L(B_n))=2n-1$.
\end{theorem}

\begin{proof}
Let $S$ be a minimum resolving set for $L(B_n)$, $n \ge 2$. Let $S^i=S \cap \tilde{C}^i$, where $1 \le i \le n$. If $|S^i|=0$ for some $i$, then $u_{i,1}$ and $u_{i, k_i}$ will have the same code; thus $|S^i| \ge 1$ for each $i \in \{1, 2, \ldots, n\}$. Next, we claim $|S^i| \ge 2$ for all but one $i \in \{1, 2, \ldots, n\}$. Assume, to the contrary, there are $S^i$ and $S^j$ with $|S^i|=1=|S^j|$ and put, without loss of generality, $i=1$ and $j=2$. We will show that the set $U=\{u_{1,1}, u_{1, k_1}, u_{2,1}, u_{2, k_2}\}$ can not be resolved by $S$. First, it is clear that every $s \in S - (S^1 \cup S^2)$ has the same distance to each vertex in $U$. Now, we show $x_1$ and $x_2$ in $S^1$ and $S^2$, respectively, can not resolve $U$. Notice $|d(u_{i,1}, x_i)-d(u_{i, k_i}, x_i)|$ is $0$ or $1$, and $d(u_{j,1}, x_i)=d(u_{j, k_j}, x_i)$ for $i \neq j$, where $i, j \in \{1,2\}$. To have any chance of resolving $U$, we must have $|d(u_{i,1}, x_i)-d(u_{i, k_i}, x_i)|=1$ for $i \in \{1,2\}$. But, in this case, a vertex from $\{u_{1,1}, u_{1, k_1}\}$ and a vertex from $\{u_{2,1}, u_{2, k_2}\}$ will necessarily share the same code. We hereby remark that $K_4$ induced by $U$ still requires three vertices to resolve as it is embedded in $\tilde{C}^1 \cup \tilde{C}^2$, in contrast to the situation in Figure \ref{Ksub}. Thus, $\dim(L(B_n)) \ge 2n-1$. Next, one can easily check that $(\cup_{i=1}^{n-1}\{u_{i,1}, u_{i, \lceil\frac{k_i}{2}\rceil+1}\}) \cup \{u_{n, \lceil\frac{k_n}{2}\rceil+1}\}$ forms a resolving set for $L(B_n)$, and thus $\dim(L(B_n)) \le 2n-1$. Therefore, $\dim(L(B_n))=2n-1$. \hfill
\end{proof}

\begin{figure}[htbp]
\begin{center}
\scalebox{0.37}{\input{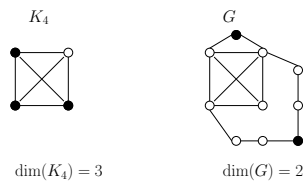}} \caption{\small{A graph $G \supset K_n$ such that $\dim(G) < \dim(K_n)$; here, the solid vertices form a minimum resolving set for each graph}}\label{Ksub}
\end{center}
\end{figure}


\section{Zero Forcing Number of Some Line Graphs}

We first define \emph{edge zero forcing} in a graph.

\begin{dfn} \label{edgezf}  Let $G$ be a connected graph of order $n \ge 2$. In analogy with the usual (vertex) color-change rule, we define \emph{ the edge color-change rule} as follows. Let each edge $e\in E(G)$ be given either the color black or the color white. A black edge $e_1$ forces the color of $e_2$ from white to black if and only if $e_2$ is the only white edge adjacent to $e_1$. An \emph{edge zero forcing set} $F\subseteq E(G)$ and the edge zero forcing number $Z_e(G)$ are then analogously defined.
\end{dfn}

\begin{obs}\label{zf-edgezf} Let $G$ be a connected graph of order $n \ge 2$. Each edge zero forcing set of $G$ corresponds to a (vertex) zero forcing set of $L(G)$;
this is a direct consequence of the definition of $L(G)$.
\end{obs}

We recall the lower bounds of the zero forcing number of a connected graph $G$ and its line graph $L(G)$.

\begin{theorem}\cite{min-degree}\label{mindegree}
For any connected graph $G$ of order $n \ge 2$, $Z(G) \ge \delta(G)$.
\end{theorem}

\begin{lemma}\cite{GbarZ} \label{Zcompletesub}
If a graph $G$ contains as a subgraph the complete graph $K_m$ ($m \ge 2$), then $Z(G) \ge Z(K_m)=m-1$.
\end{lemma}

\begin{prop} \cite{AIM} \label{line_complete}
For the complete graph $K_n$ of order $n \ge 4$, $Z(L(K_n)) = \frac{1}{2}(n^2-3n+4)$.
\end{prop}

\begin{prop}\label{ZMAX}
For any connected graph $G$ of order $n \ge 2$, $$\Delta(G) -1 \le Z(L(G)) \le |E(G)|-(\delta(G)-1),$$ and both bounds are sharp.
\end{prop}

\begin{proof}
The lower bound immediately follows from Lemma~\ref{Zcompletesub}, since $L(G)$ contains $K_{\Delta(G)}$ as a subgraph. For the sharpness of the lower bound, take $G=K_{1, n-1}$.\\

The upper bound obviously holds for $\delta(G)\leq 2$. So, let $v_0\in V(G)$ be a vertex of degree $\delta=\delta(G)\geq 3$, and let $v_1, v_2, \ldots, v_\delta$ be the vertices adjacent to $v_0$.
We claim that $E(G) - \{v_0v_2, \ldots, v_0v_\delta\}$ forms an edge zero forcing set for $G$. The claim follows from the observation that, for each $i\in I=\{2, \ldots, \delta\}$, there exists a black edge $e$ incident
with $v_i$ and adjacent to $v_0v_j$ ($j\in I$) exactly when $j=i$ (see Figure~\ref{edgezero}): this is because the edges incident with $v_i$ ($i \in I$) and not satisfying the requirement of the observation lie in the set $A_i=\{v_iv_j: j=0 \mbox{ or } j\in I -\{i\}\}$, and $|A_i|\leq \delta-1< \deg_G(v_i)$. For the sharpness of the upper bound, take $G=K_n$ (see Proposition \ref{line_complete}). \hfill
\end{proof}

\begin{figure}[htbp]
\begin{center}
\scalebox{0.4}{\input{edgezero.pstex_t}} \caption{\small{Here, $\deg_G(v_0)=\delta(G)=4$ and the edges drawn in dotted lines are excluded from the edge zero forcing set of $G$. Note that $A_4=\{v_4v_0, v_4v_2, v_4v_3\}$ and, a priori, an edge in $A_4$ cannot force. But the edge $v_4v_1$ can force the edge $v_4v_0$ to black.}}\label{edgezero}
\end{center}
\end{figure}

The \emph{path cover number} $P(G)$ of $G$ is the minimum number of vertex disjoint paths, occurring as induced subgraphs of $G$, that cover all the vertices of $G$.

\begin{theorem} \cite{AIM, pathcover} \label{pathcover}
\begin{itemize}
\item[(a)] \cite{pathcover} For any graph $G$, $P(G) \le Z(G)$.
\item[(b)] \cite{AIM} For any tree $T$, $P(T) = Z(T)$.
\end{itemize}
\end{theorem}

\begin{theorem} \cite{AIM}\label{ZlineT}
For any nontrivial tree $T$, $Z(L(T))=\sigma(T)-1$.
\end{theorem}

\begin{prop} \cite{dimZ, Row} \label{ObsZ} Let $G$ be a connected graph of order $n \ge 2$. Then
\begin{itemize}
\item[(a)] $Z(G)=1$ if and only if $G=P_n$,
\item[(b)] $Z(G)=n-1$ if and only if $G=K_n$.
\end{itemize}
\end{prop}

\begin{theorem} \cite{Z+e} \label{Zedgeremoval}
Let $G$ be any graph. Then
\begin{itemize}
\item[(a)] For $v \in V(G)$, $Z(G)-1 \le Z(G-v) \le Z(G)+1$.
\item[(b)] For $e \in E(G)$, $Z(G)-1 \le Z(G-e) \le Z(G)+1$.
\end{itemize}
\end{theorem}

Next, we compare $Z(G)$ and $Z(L(G))$.

\begin{theorem} \label{ZLtree2}
For any connected graph $G$, $Z(G) \leq 2 Z(L(G))$.
\end{theorem}

\begin{proof}
Let $Z$ be a minimum edge zero-forcing set for $G$. Form a set $Z'$ in $V(G)$ by taking both endpoints of each edge which appears in $Z$. Notice that $|Z'| \leq 2|Z| = 2 Z(L(G))$. We claim that $Z'$ is a zero-forcing set for $G$.\\

Notice that if an edge is black, then both of its endpoints are black in $G$, but the converse is not necessarily true. In the first iteration (i.e., one global application of the color-change rule) of edge zero-forcing, the black edges in $Z$ force other edges to become black. Each edge which is forced to become black is adjacent to a black edge, so at least one endpoint of each edge that is forced black was already black in $G$. If both endpoints were black in $G$, then nothing changes in $G$. Suppose one endpoint, say $u$, was black, and the other endpoint, say $v$, was white. Then $uv$ is forced by a neighboring black edge to become black. Without loss of generality, the neighboring black edge was $xu$ for some $x \in V(G)$. Every edge other than $uv$ adjacent to $xu$ must have been black, so all of their endpoints were black in $G$. In particular, every neighbor of $u$ other than $v$ was already black, so $v$ is forced to turn black in $G$ in this iteration. Thus, we maintain the property that if an edge is black, then both of its endpoints are black after the same number of iterations.\\

Since eventually every edge is black, and $G$ is a connected nontrivial graph, eventually every vertex of $G$ is black. Thus, $Z'$ is a zero-forcing set for $G$.~\hfill
\end{proof}

Next, we give an example of a tree $T$ satisfying $Z(T) \neq Z(L(T))$, which is worth mentioning, since $\dim(T)=\dim(L(T))$ for any tree $T$ given Theorem \ref{lineT} and the fact $L(P_n)=P_{n-1}$.

\begin{rem}
The tree $T$ of Figure \ref{treeDZ} satisfies $Z(T)=P(T)=7$, and one can easily check that $Z(L(T))=11$; here $T$ of Figure \ref{treeDZ} can be viewed as a tree obtained by attaching $4$ copies of a subtree on 6 vertices at the central vertex. If $T'$ is a tree obtained by attaching $k \ge 3$ branches of the subtree at the central vertex, one can easily check that $Z(T')=P(T')=2k-1$ and $Z(L(T'))=3k-1$.
\end{rem}

\begin{figure}[htbp]
\begin{center}
\scalebox{0.37}{\input{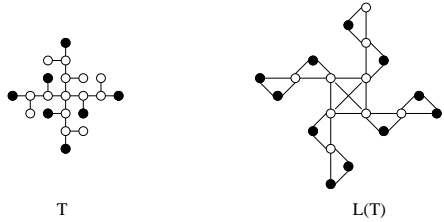}} \caption{\small{A tree $T$ satisfying $Z(T) \neq Z(L(T))$, where the solid vertices form a minimum zero forcing set for $T$ and $L(T)$, respectively.}}\label{treeDZ}
\end{center}
\end{figure}

\begin{theorem} \label{ZLtree1}
For any nontrivial tree $T$, $Z(T) \le Z(L(T))$.
\end{theorem}

\begin{proof}
Let $T$ be a tree of order $n \ge 2$. By Theorem \ref{ZlineT} and (b) of Theorem \ref{pathcover}, it suffices to show
\begin{equation}\label{eqtree}
P(T)\leq \sigma(T)-1 .
\end{equation}
First, notice that (\ref{eqtree}) holds for $T=P_2$. Notice also that every tree $T$ may be obtained from $P_2$ by attaching finitely many pendant edges. Let $T=T'+e$ denote the vertex sum of $T'$ and a disjoint copy of $P_2$, and assume (\ref{eqtree}) holds for $T'$. If $\sigma(T)=\sigma(T')$, then $e=uv$ is ``attached" to an end-vertex $v'$ of $T'$; let's say $v$ is identified with $v'$. Take any path cover $C'$ of $T'$ with $|C'|=P(T') \leq \sigma(T')-1$. Since $v'$ is an end-vertex in $T'$, it must be either the first or the last vertex in a path $Q' \in C'$. Let $Q=Q'+e$, where the vertex sum is formed by identifying $v'$ in $T'$ with $v$ of $e$. Then $C=(C' - \{Q'\}) \cup \{Q\}$ is a pathcover for $T$; hence (\ref{eqtree}) holds for $T$. If $\sigma(T)=\sigma(T')+1$, then $C=C' \cup \{u\}$ suffices as a path cover for $T$ showing that (\ref{eqtree}) holds for $T$.\hfill
\end{proof}

Next, we recall a result which is useful in establishing the lower bound for the zero forcing number of some line graphs.

\begin{prop}\cite{AIM} \label{wheelAIM}
Let $G$ be a graph of order $n \ge 2$. If $G$ contains a Hamiltonian path, then $mr(L(G))=n-2$.
\end{prop}

\begin{prop} \label{open_zero_line}
Let a connected graph $G$ either have order $n \leq 4$ or contain a Hamiltonian path and satisfies $|E(G)| \ge 2(n-2)$; then $Z(G) \le Z(L(G))$.
\end{prop}

\begin{proof}
If $2 \le n \le 4$, for any connected graph $G$ (not necessarily containing a Hamiltonian path), one can check all cases to see that $Z(G) \le Z(L(G))$. So, let $G$ have order $n\geq 5$ and contain a Hamiltonian path.  Notice that $Z(G) \le n-1$ by connectedness of $G$; note also that $Z(G)=n-1$ if and only if $G=K_n$ by (b) of Proposition \ref{ObsZ}. Since $Z(K_n) \le Z(L(K_n))$ by Proposition \ref{line_complete}, it suffices to consider $Z(G)\leq n-2$. By Proposition \ref{wheelAIM}, $Z(L(G)) \ge M(L(G)) \ge |V(L(G))|-(n-2)=|E(G)|-(n-2)$. The assertion $Z(G) \le Z(L(G))$ is satisfied by requiring that $|E(G)|-(n-2) \ge n-2$, which is equivalent to the hypothesis $|E(G)| \ge 2(n-2)$.~\hfill
\end{proof}


\subsection{Wheel Graphs}

We first determine the zero forcing number of the wheel graph $W_{1,n}$ for $n \ge 3$.

\begin{prop}
For $n \ge 3$, $Z(W_{1,n})=3$.
\end{prop}

\begin{proof}
Since $\delta(W_{1,n})=3$, $Z(W_{1,n}) \ge 3$ by Theorem \ref{mindegree}. It is easy to see that $\{v, w_1, w_n\}$ forms a zero forcing set of $W_{1,n}$ (see (A) of Figure \ref{lemmawheel}): $w_1 \rightarrow w_2 \rightarrow \ldots \rightarrow w_{\lfloor \frac{n+1}{2}\rfloor}$ and $w_n \rightarrow w_{n-1} \rightarrow \ldots \rightarrow w_{\lceil \frac{n+1}{2}\rceil}$. So, $Z(W_{1,n}) \le 3$. Therefore, $Z(W_{1,n})=3$ for $n \ge 3$. \hfill
\end{proof}

\begin{theorem}\label{zerolinewheel}
For $n \ge 3$, $Z(L(W_{1,n}))= n+1$.
\end{theorem}

\begin{proof}
Let $n \ge 3$. Since $W_{1,n}$ contains a Hamiltonian path, $mr(L(W_{1,n}))=n-1$ by Proposition \ref{wheelAIM}. Since $Z(G) \geq M(G)$ and $M(G)= |V(G)|-mr(G)$ for any graph $G$, we get $Z(L(W_{1,n}))\geq n+1$ as $L(W_{1,n})$ has order $2n$. On the other hand, take the labeling given in Figure \ref{lemmawheel}; we see that $S=\{\ell_0, \ell_1\} \cup \{u_i \mid 1 \le i \le n-1\}$ forms a zero forcing set for $L(W_{1,n})$: $\ell_1 \rightarrow \ell_2 \rightarrow \ldots \rightarrow \ell_{n-1} \rightarrow u_0$. Therefore, $Z(L(W_{1,n}))=n+1$ for $n \ge 3$. \hfill
\end{proof}

\begin{figure}[htbp]
\begin{center}
\scalebox{0.37}{\input{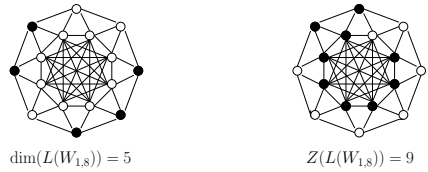}} \caption{\small{Minimum resolving set and minimum zero forcing set (solid vertices in each case) of the line graph of the wheel graph $W_{1,8}$}}\label{wheeldimZ}
\end{center}
\end{figure}

The next result shows that the above result is ``edge-critical" in some sense.

\begin{prop}
Let $G=L(W_{1,n})$ for $n \ge 3$. Then $Z(G-e)=n$ for any edge $e \in E(L(W_{1,n}))$.
\end{prop}

\begin{proof}
Let $H=L(W_{1,n}) -e$, where $e$ is an edge of $L(W_{1,n})$. We will show that there exists a zero forcing set $S$ for $H$ with $|S|=n$, and thus $Z(H) \le n$. If $e=\ell_iu_i$, say $i=0$, then $S=\{\ell_0, \ell_{n-1}\} \cup \{u_i \mid 1 \le i \le n-2\}$ forms a zero forcing set: $\ell_0 \rightarrow \ell_1 \rightarrow \ell_2 \rightarrow \ldots \rightarrow \ell_{n-2} \rightarrow u_{n-1} \rightarrow u_0$. If $e=\ell_i\ell_{i+1}$, say $i=0$, then $S=\{\ell_1\} \cup \{u_i \mid 1 \le i \le n-1\}$ forms a zero forcing set: (i) $\ell_1 \rightarrow \ell_2 \rightarrow \ldots \rightarrow \ell_{n-1}$; (ii) $u_2 \rightarrow u_0 \rightarrow \ell_0$. If $e=u_iu_{i+1}$, say $i=0$, then $S=\{\ell_0, \ell_1\} \cup \{u_i \mid 1 \le i \le n-2\}$ forms a zero forcing set: (i) $u_1 \rightarrow u_{n-1}$; (ii) $\ell_1 \rightarrow \ell_2 \rightarrow \ldots \rightarrow \ell_{n-1} \rightarrow u_0$. If $e=u_iu_{j}$ with $|i-j| \neq 1$ (mod $n$), say $i=0$, then $S=\{\ell_0, \ell_{n-1}\} \cup (\{u_i \mid 0 \le i \le n-1\} - \{u_1, u_j\})$ forms a zero forcing set: (i) $u_0 \rightarrow u_1$; (ii) $\ell_0 \rightarrow \ell_1$; (iii) $u_1 \rightarrow u_j$; (iv) $\ell_1 \rightarrow \ell_2 \rightarrow \ldots \rightarrow \ell_{n-2}$. So, in each case, we have $Z(H) \le n$. On the other hand, $Z(H) \ge n$ by (b) of Theorem \ref{Zedgeremoval} and Theorem \ref{zerolinewheel}. Thus, $Z(H)=n$. \hfill
\end{proof}


\subsection{Bouquet of Circles}

\begin{theorem} \cite{iteration}\label{bouquetZ}
Let $B_n=(k_1+1, k_2+1, \ldots, k_n+1)$ be a bouquet of $n \ge 2$ circles with
a cut-vertex. Then $Z(B_n)=n+1$.
\end{theorem}

\begin{theorem} \label{ZeroLBouquet}
Let $B_n=(k_1+1, k_2+1, \ldots, k_n+1)$ be a bouquet of $n \ge 2$ circles with
a cut-vertex. Then $Z(L(B_n))=2n-1$.
\end{theorem}

\begin{proof}
Since $\Delta(B_n)=2n$, $Z(L(B_n)) \ge 2n-1$ by Proposition \ref{ZMAX}. On the other hand, take the labeling given in Figure \ref{bouquetlabeling}; since $(\cup_{i=1}^{n-1}\{u_{i,1}, u_{i, 2}\}) \cup \{u_{n,1}\}$ forms a zero forcing set for $L(B_n)$, $Z(L(B_n)) \le 2n-1$. Thus, $Z(L(B_n))=2n-1$.\hfill
\end{proof}

\begin{figure}[htbp]
\begin{center}
\scalebox{0.4}{\input{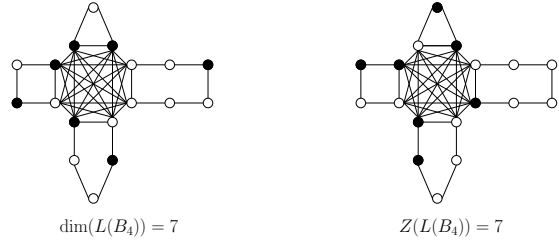}} \caption{\small{Minimum resolving set and minimum zero forcing set (solid vertices in each case) of the line graph of the bouquet of four circles $B_4=(3,4,5,6)$}}\label{bouquetdimZ}
\end{center}
\end{figure}


\section{Comparison and Open Problems}

In this section, we compare the metric dimension with the zero forcing number of a line graph by demonstrating a couple of inequalities between the two parameters. We also mention some open problems. First, recall some results obtained in \cite{dimZ}.

\begin{theorem} \cite{dimZ}
\begin{itemize}
\item[(a)] For any tree $T$, $\dim(T) \leq Z(T)$,
\item[(b)] For any tree $T$ and an edge $e \in E(\overline{T})$, $\dim(T+e) \le Z(T+e)+1$.
\end{itemize}
\end{theorem}

As an immediate consequence of Theorem \ref{lineT} and Theorem \ref{ZlineT}, we have the following

\begin{cor}\label{comp0}
For any tree $T$, $\dim(L(T)) \le Z(L(T))$.
\end{cor}

If the order of a connected graph $G$ is $4$ or less, one can check all cases to see that $\dim(L(G)) \le Z(L(G))$. Now, we show that $\dim(L(G)) \le Z(L(G))$ is satisfied for another class of graphs. 

\begin{prop} \label{comp2}
Suppose that a connected graph $G$ of order $n\geq 5$ contains a Hamiltonian path and satisfies $|E(G)| \ge 2(n-2)$. Then $\dim(L(G)) \le Z(L(G))$.
\end{prop}

\begin{proof}
Let $G$ satisfying the hypotheses be given. Notice that $\dim(L(G)) \le n-2$ by Theorem \ref{linebounds}. By Proposition \ref{wheelAIM}, $Z(L(G)) \ge M(L(G)) \ge |V(L(G))|-(n-2)=|E(G)|-(n-2)$. The assertion $\dim(L(G)) \le Z(L(G))$ is satisfied by requiring that $|E(G)|-(n-2)\ge n-2$, which is equivalent to the hypothesis $|E(G)| \ge 2(n-2)$.~\hfill
\end{proof}

Proposition \ref{comp2}, together with the fact $2 \cdot |E(G)|=\sum_{v\in V(G)}\deg(v)=|V(G)| \cdot (\mbox{average degree of }G)$, implies the following

\begin{cor}
Let a connected graph $G$ have order $n \ge 5$ and contain a Hamiltonian path. If $5 \le n \le 8$ and the average degree of $G$ is at least $3$, then $\dim(L(G)) \le Z(L(G))$; if $n \ge 9$ and the average degree of $G$ is at least $4$, then $\dim(L(G)) \le Z(L(G))$.
\end{cor}

Next, we give a table on metric dimension and zero forcing number of some line graphs. We denote by $T$ a tree, $K_{s,t}$ the complete bi-partite graph of order $s+t$, $W_{1,n}$ the wheel graph of order $n+1$, $B_n$ the bouquet of $n$ circles. Further, $[n]$ next to a formula indicates that the formula can be found in reference $[n]$.

\begin{center}
\begin{tabular}{|r| |r@{}l p{2in}|}
\hline
\multicolumn{1}{|c||}{$L(G)$}&
\multicolumn{1}{c|}{Metric Dimension}&
\multicolumn{1}{c|}{Zero Forcing Number}
\\  \hline
\multicolumn{1}{|c||}{$L(T)$, $T \neq P_n$}&
\multicolumn{1}{c|}{$\sigma(T)-ex(T)$ \cite{line}}&
\multicolumn{1}{c|}{$\sigma(T)-1$ \cite{AIM}}
\\ \hline
\multicolumn{1}{|c||}{$L(K_n)$, $n \ge 6$}&
\multicolumn{1}{c|}{$\lceil\frac{2n}{3} \rceil$ \cite{base}}&
\multicolumn{1}{c|}{$\frac{1}{2}(n^2-3n+4)$ \cite{AIM}}
\\  \hline
\multicolumn{1}{|c||}{$L(K_{s,t})$, $s,t \ge 2$}&
\multicolumn{1}{c|}{$\left\{
\begin{array}{ll}
\lfloor\frac{2(s+t-1)}{3}\rfloor & \mbox{ if } s \le t \le 2s-1\\
t-1 & \mbox{ if } t \ge 2s
\end{array} \right.$ \cite{cartesian}}&
\multicolumn{1}{c|}{$st-s-t+2$ \cite{AIM}}
\\  \hline
\multicolumn{1}{|c||}{$L(W_{1,n})$, $n \ge 3$}&
\multicolumn{1}{c|}{
$\left\{
\begin{array}{ll}
3 & \mbox{ if } n=3,4\\
4 & \mbox{ if }n=5\\
n-\lceil \frac{n}{3}\rceil & \mbox{ if } n \ge 6
\end{array} \right.$
}&
\multicolumn{1}{c|}{$n+1$}
\\  \hline
\multicolumn{1}{|c||}{$L(B_n)$, $n \ge 2$}&
\multicolumn{1}{c|}{$2n-1$}&
\multicolumn{1}{c|}{$2n-1$}
\\  \hline
\end{tabular}
\end{center}

We conclude this paper with some open problems.\\

\textbf{Problem 1.} We proved in Theorem \ref{ZLtree1} that $Z(T) \le Z(L(T))$ for any nontrivial tree $T$. We also proved in Proposition \ref{open_zero_line} that $Z(G) \le Z(L(G))$ for a graph $G$ such that $|E(G)| \ge 2|V(G)|-4$ and $G$ contains a Hamiltonian path. For a general graph $G$, we proved in Theorem \ref{ZLtree2} that $Z(G) \le 2Z(L(G))$. We conjecture that $Z(G) \le Z(L(G))$ for any $G$.\\

\textbf{Problem 2.} It is stated in Corollary \ref{comp0} that $\dim(L(T)) \le Z(L(T))$ for a tree $T$. We proved in Proposition \ref{comp2} that $\dim(L(G)) \le Z(L(G))$ for a graph $G$ such that $|E(G)| \ge 2|V(G)|-4$ and $G$ contains a Hamiltonian path. We conjecture that $\dim(L(G)) \le Z(L(G))$ for any $G$.\\

\textbf{Problem 3.} A characterization of a tree $T$ such that $\dim(T)=Z(T)$ is given in \cite{dimZ}. What about characterizing $\dim(G)=Z(G)$ for other classes of graphs?\\

\textbf{Problem 4.} We proved in Theorem \ref{DimLBouquet} and Theorem \ref{ZeroLBouquet} that $\dim(L(B_n))=Z(L(B_n))$. What about characterizing $\dim(L(G))=Z(L(G))$ for other classes of graphs?\\

\textit{Acknowledgement.} The authors thank the anonymous referees for some helpful comments and suggestions which improved the paper.



\begin{thebibliography}{99}

\bibitem{AIM} AIM Minimum Rank - Special Graphs Work Group (F. Barioli, W. Barrett, S. Butler, S.M. Cioab\u{a}, D. Cvetkovi\'{c}, S.M. Fallat, C. Godsil, W. Haemers, L. Hogben, R. Mikkelson, S. Narayan, O. Pryporova, I. Sciriha, W. So, D. Stevanovi\'{c}, H. van der Holst, K. Vander Meulen, A.W. Wehe). Zero forcing sets and the minimum rank of graphs. {\it Linear Algebra Appl. } {\bf{428}}, No. 7 (2008) 1628-1648.

\bibitem{base} R.F. Bailey and P.J. Cameron, Base size, metric dimension and other invariants of groups and graphs. \emph{Bull. Lond. Math. Soc.} \textbf{43}, No. 2 (2011) 209-242.

\bibitem{pathcover} F. Barioli, W. Barrett, S.M. Fallat, H.T. Hall, L. Hogben, B. Shader, P. van den Driessche and H. van der Holst, Zero forcing parameters and minimum rank problems. \emph{Linear Algebra Appl.} \textbf{433}, No. 2 (2010) 401-411.

\bibitem{min-degree} A. Berman, S. Friedland, L. Hogben, U.G. Rothblum and B. Shader, An upper bound for the minimum rank of a graph. \textit{Linear Algebra Appl.} \textbf{429}, No. 7 (2008) 1629-1638.

\bibitem{wheel2} P.S. Buczkowski, G. Chartrand, C. Poisson and P. Zhang, On $k$-dimensional graphs and their bases. \emph{Period. Math. Hungar.} \textbf{46}, No. 1 (2003) 9-15.

\bibitem{cartesian} J. C\'{a}ceres, C. Hernando, M. Mora, I.M. Pelayo, M.L. Puertas, C. Seara and D.R. Wood, On the metric dimension of Cartesian products of graphs. \textit{SIAM J. Discrete Math.} \textbf{21}, No. 2 (2007) 423-441.

\bibitem{CEJO} G. Chartrand, L. Eroh, M.A. Johnson and O.R. Oellermann, Resolvability in graphs and the metric dimension of a graph. \textit{Discrete Appl. Math.} \textbf{105}, No. 1-3 (2000) 99-113.

\bibitem{MDsurvey} G. Chartrand and P. Zhang, The theory and applications of resolvability in graphs. A Survey. \textit{Congr. Numer.} {\bf{160}} (2003) 47-68.

\bibitem{iteration} K.B. Chilakamarri, N. Dean, C.X. Kang and E. Yi, Iteration index of a zero forcing set in a graph. \textit{Bull. Inst. Combin. Appl.} {\bf{64}} (2012) 57-72.


\bibitem{Z+e} C.J. Edholm, L. Hogben, M. Huynh, J. LaGrange and D.D. Row, Vertex and edge spread of zero forcing number, maximum nullity, and minimum rank of a graph. \textit{Linear Algebra Appl.} \textbf{436}, No. 12 (2012) 4352-4372. 

\bibitem{dimZ} L. Eroh, C.X. Kang and E. Yi, A comparison between the metric dimension and zero forcing number of trees and unicyclic graphs. \textit{submitted}.

\bibitem{GbarZ} L. Eroh, C.X. Kang and E. Yi, On zero forcing number of graphs and their complements. \textit{submitted}.

\bibitem{ZFsurvey} S.M. Fallat and L. Hogben. The minimum rank of symmetric matrices described by a graph: a survey. \textit{Linear Algebra Appl.} \textbf{426}, No. 2-3 (2007) 558-582.

\bibitem{ZFsurvey2} S.M. Fallet and L. Hogben, Variants on the minimum rank problem: a survey II. \textit{arXiv:1102.5142v1}.

\bibitem{line} M. Feng, M. Xu and K. Wang, On the metric dimension of line graphs. \textit{Discrete Appl. Math.} \textbf{161}, No. 6 (2013) 802-805.

\bibitem{NPcompleteness} M.R. Garey and D.S. Johnson, \emph{Computers and intractability: A guide to the theory of NP-completeness.} Freeman, New York (1979).

\bibitem{HM} F. Harary and R.A. Melter, On the metric dimension of a graph. \textit{Ars Combin.} {\bf{2}} (1976) 191-195.

\bibitem{new} C. Hernando, M. Mora, I.M. Pelayo, C. Seara and D.R. Wood, Extremal graph theory for metric dimension and diameter. \textit{Electron. J. Combin.} \textbf{17}, No. 1 (2010) Research Paper 30, 28~pp.

\bibitem{proptime} L. Hogben, M. Huynh, N. Kingsley, S. Meyer, S. Walker and M. Young, Propagation time for zero forcing on a graph. \textit{Discrete Appl. Math.} \textbf{160}, No. 13-14 (2012) 1994-2005.

\bibitem{bouquetD} H. Iswadi, E.T. Baskoro, A.N.M. Salman and R. Simanjuntak, The metric dimension of amalgamation of cycles. \emph{Far East J. Math. Sci.} \textbf{41}, No. 1 (2010) 19-31.

\bibitem{landmarks} S. Khuller, B. Raghavachari and A. Rosenfeld, Landmarks in graphs. \textit{Discrete Appl. Math.} \textbf{70}, No. 3 (1996) 217-229.

\bibitem{bouquet} J. Llibre and M. Todd, Periods, Lefschetz numbers and entropy for a class of maps on a bouquet of circles.  \textit{J. Difference Equ. Appl.} \textbf{11}, No.12 (2005) 1049-1069.

\bibitem{massey} W.S. Massey, {\it Algebraic topology: An introduction.} Graduate Texts in Mathematics \textbf{56}, Springer-Verlag (1989).

\bibitem{PoZh} C. Poisson and P. Zhang, The metric dimension of unicyclic graphs. \emph{J. Combin. Math. Combin. Comput.} \textbf{40} (2002) 17-32.

\bibitem{Row} D.D. Row, A technique for computing the zero forcing number of a graph with a cut-vertex. \textit{Linear Algebra Appl.} \textbf{436}, No. 12 (2012) 4423-4432.

\bibitem{MathZ} A. Seb\"{o} and E. Tannier, On metric generators of graphs. \emph{Math. Oper. Res.} \textbf{29}, No. 2 (2004) 383-393.

\bibitem{wheels} B. Shanmukha, B. Sooryanarayana and K.S. Harinath, Metric dimension of wheels. \textit{Far East J. Appl. Math.} \textbf{8}, No. 3 (2002) 217-229.

\bibitem{Slater2} P.J. Slater, Dominating and reference sets in a graph. \emph{J. Math. Phys. Sci.} {\bf{22}}, No. 4 (1988) 445-455.

\bibitem{Slater} P.J. Slater, Leaves of trees. \textit{Congr. Numer.} {\bf{14}} (1975) 549-559.

\bibitem{Whitney} H. Whitney, Congruent graphs and the connectivity of graphs. \emph{Amer. J. Math.} \textbf{54} (1932) 150-168.


\end{thebibliography}
\end{document}